\documentclass[11pt]{article}
\usepackage{graphics}
\usepackage{amsfonts}
\usepackage{amsmath, amssymb, amsthm}
\usepackage{tikz-cd}
\usepackage[T1]{fontenc}
\usepackage[utf8]{inputenc}
\usepackage{mathtools}
\usepackage{comment}
\usepackage{scalerel}
\usepackage{diagbox}

\usepackage[normalem]{ulem} 
\usepackage{xcolor}         

\def\Box{\vcenter{\vbox{\hrule\hbox{\vrule
     \vbox to 8.8pt{\hbox to 10pt{}\vfill}\vrule}\hrule}}}

\newcommand{\Z}{{\mathbb Z}}

\newcommand{\F}{\mathbb{F}}

\newcommand{\la}{\lambda}

\newcommand{\Fp}{\F_p}
\newcommand{\Fq}{\F_q}

\newcommand{\NF}{\text{NF}}

\newcommand{\es}{\emptyset}
\newcommand{\lan}{\langle}
\newcommand{\ran}{\rangle}

\newtheorem{thm}{Theorem}[section]
\newtheorem{lemma}[thm]{Lemma}

\newtheorem{prop}[thm]{Proposition}

\newtheorem{res}[thm]{Result}

\newtheorem{rem}[thm]{Remark}

\newcommand\rwh[1]{\arraycolsep=0pt\relax%
\begin{array}{c}
\stretchto{
  \scaleto{
    \scalerel*[\widthof{\ensuremath{#1}}]{\kern-.5pt\bigwedge\kern-.5pt}
    {\rule[-\textheight/2]{1ex}{\textheight}} 
  }{\textheight} %
}{0.5ex}\\           
#1\\                 
\rule{-1ex}{0ex}
\end{array}
}

\usepackage{xspace}
\usepackage{xparse}
\DeclareDocumentCommand{\NF}{o m m}%
{%
\IfValueTF{#1}
 {\ensuremath{#1\textsc{-nf}(#2,#3)}\xspace}
 {\ensuremath{\textsc{nf}(#2,#3)}\xspace}%
}

\title{Cyclotomic construction of $\lambda$-fold near-factorizations of cyclic groups}

\author{
Shuxing Li\footnote{Department of Mathematical Sciences, University of Delaware, Newark, DE 19716, USA, email: {\tt shuxingl@udel.edu}} \quad
Koji Momihara\footnote{Division of Natural Science, Faculty of Advanced Science and Technology, Kumamoto University, 2-40-1 Kurokami, Kumamoto 860-8555, Japan, email: {\tt momihara@educ.kumamoto-u.ac.jp}} 
}

\begin{document}

\date{}
\maketitle

\begin{abstract}
The study of near-factorizations of finite groups dates back to the 1950s. Recently, this topic has attracted renewed attention, and the concept has been extended to $\lambda$-fold near-factorizations, in which each non-identity group element appears exactly $\lambda \ge 1$ times. This paper presents a cyclotomic construction of $\lambda$-fold near-factorizations in the cyclic group $\Fp$, where $p = 4n^4 + 12n^2 + 1$ is prime for $n \ge 1$.


\end{abstract}

\section{Introduction}
\label{sec-introduction}

Factoring a group into products of its subsets has been a vibrant area of research, with applications in combinatorics, coding theory, functional analysis, and number theory \cite{Sza,SS}. Let $G$ be a finite group, and let $S$ and $T$ be subsets of $G$. The pair $(S,T)$ is called a \emph{near-factorization of $G$} if every non-identity element $g \in G$ can be uniquely expressed as a product $g = st$, where $s \in S$ and $t \in T$. The study of near-factorizations of finite groups was initiated by de Bruijn in the 1950s \cite{DB} and has since developed into a long-standing line of research on both finite abelian and nonabelian groups \cite{BHS, CGHK, Pech03, Pech04}. Recently, there have been further advances in the study of near-factorizations \cite{KMS, KPS}. Notably, the theoretical and computational results in \cite{KMS} indicate that no nontrivial near-factorizations exist for finite abelian noncyclic groups of order less than 200. This observed sparsity of near-factorizations in finite abelian groups has motivated Kreher, Li, and Stinson \cite{KLS} to propose a natural generalization of the concept.

Let $S$ and $T$ be subsets of a finite group $G$, and let $\lambda$ be a positive integer. The pair $(S, T)$ is called a \emph{$\lambda$-fold near-factorization of $G$} if every non-identity element $g \in G$ can be expressed as a product $g = st$ with $s \in S$ and $t \in T$ in exactly $\lambda$ distinct ways. Clearly, a $1$-fold near-factorization is simply a near-factorization of $G$. Compared to near-factorizations, allowing each non-identity element of $G$ to be represented multiple times leads to much richer constructions for $\lambda$-fold near-factorizations in finite abelian groups. Please refer to \cite{KLS} for a comprehensive treatment.  

In this paper, we will use the standard group ring notation. For a gentle introduction to group rings, please refer to \cite[Section 1]{JL24}. For a finite group $G$ with identity $e$, using the group ring notation, $(S,T)$ is a $\lambda$-fold near-factorization of $G$, if
$$
ST= \lambda(G-e) \mbox{\quad in $\Z[G]$}.
$$
Furthermore, if $|S|=s$ and $|T| =t$, we say $(S,T)$ is an $(s,t)$-$\lambda$-fold near-factorization of $G$, denoted by \NF[(s,t)]{G}{\lambda}. Note that an \NF[(s,t)]{G}{\lambda} exists only if 
$$
st = \lambda(|G| -1).
$$
While the $\lambda$-fold near-factorization has only been proposed very recently \cite{KLS}, a special subfamily of $\lambda$-fold near-factorization has been studied in the context of coding theory and cryptography \cite{PS}. Let $D_1, D_2, \dots, D_m$ be mutually disjoint $k$-subsets of $G$, where $m \ge 2$, and let $\lambda$ be a positive integer. Then $\{D_1, D_2, \dots, D_m\}$ is a $(v, m, k, \la)$-\emph{strong external difference family} (SEDF) in $G$ if
$$
D_j \sum_{\substack{1 \le i \le m \\ i \ne j}} D_i^{(-1)}=\la(G-1) \quad \mbox{in $\Z[G]$ for each $j$ satisfying $1 \le j \le m$},
$$
where $D_i^{(-1)}$ is the group ring element corresponding to the subset $\{ x^{-1} \mid x \in D_i \}$. Note that an SEDF in $G$ consisting of exactly two subsets is equivalent to a $\lambda$-fold near-factorization of $G$. Specifically, $\{D_1, D_2\}$ is a $(v,k,2,\la)$-SEDF in $G$ if and only if $(D_1,D_2^{(-1)})$ is a $\NF[(k,k)]{G}{\lambda}$. SEDFs were proposed in \cite{PS} as a combinatorial counterpart of the optimal algebraic manipulation detection code \cite{CDF+}, which has applications in cryptography. For more results concerning SEDFs, please refer to \cite{BJWZ,HP,JL19,PS} and the references therein.

From the perspective of SEDFs, there have been several cyclotomic constructions generating $\lambda$-fold near-factorizations of the additive group of finite fields. Let $N>1$ be an integer. Let $q$ be a prime with $q \equiv 1 \pmod N$. Let $\alpha$ be a fixed primitive element of the finite field $\Fq$. Set $C_0^N=\lan \alpha^N \ran$ to be the unique index $N$ subgroup of the multiplicative cyclic group $\Fq^*$. For $0 \le i \le N-1$, define $C_i^N$ to the multiplicative coset $\alpha^i\lan \alpha^N \ran$. The cosets $\{C_i^N \mid 0 \le i \le N-1\}$ are \emph{cyclotomic cosets of order $N$ over $\Fq$}, which serve as building blocks of cyclotomic constructions. We note that cyclotomic cosets over $\Fq$ depends on the choice of the primitive element $\alpha$ of $\Fq$.

Specifically, $\lambda$-fold near-factorizations have been constructed using order $2$ \cite{CD,HP}, order $4$ \cite{BJWZ}, and order $6$ \cite{BJWZ} cyclotomic cosets. We summarize these results as follows.

\begin{res}
\label{res-prior}
\begin{itemize}
\item[(1)] {\rm (\cite[Proposition 29]{CD}, \cite[Theorem 5.4]{HP})} Let $q$ be a prime power with $q \equiv 1 \pmod4$. Then $(C_0^2,C_1^2)$ is a $\NF[(\frac{q-1}{2},\frac{q-1}{2})]{\Fq}{\frac{q-1}{4}}$.  
\item[(2)] {\rm (\cite[Theorem 4.3]{BJWZ})} Let $p$ be a prime with $p=16n^2+1$ for $n \ge 1$. Then $(C_0^4,C_2^4)$ is a $\NF[(\frac{p-1}{4},\frac{p-1}{4})]{\Fp}{\frac{p-1}{16}}$.  
\item[(3)] {\rm (\cite[Theorem 4.6]{BJWZ})} Let $p$ be a prime with $p=108n^2+1$ for $n \ge 1$. Then $(C_0^6,C_3^6)$ is a $\NF[(\frac{p-1}{6},\frac{p-1}{6})]{\Fp}{\frac{p-1}{36}}$.  
\end{itemize}
\end{res}

\begin{rem}
Result~\ref{res-prior}(2) describes the construction presented in \cite[Theorem 4.3]{BJWZ}, where $\NF[(\frac{q-1}{4},\frac{q-1}{4})]{\Fq}{\frac{q-1}{16}}$ was obtained for each prime power $q$ with $q=16n^2+1$ for $n \ge 1$. Note that if $q=16n^2+1$ is a prime power, then $q$ must be a prime because Catalan's conjecture is known to hold~\cite{Mih}.
\end{rem}

In this paper, we utilize order $8$ cyclotomic cosets to construct $\lambda$-fold near-factorizations of cyclic groups with prime order. Our construction can be summarized as follows.

\begin{thm}
\label{thm-main}
Let $u$ be a nonzero square such that $p=4u^{2}+12u+1$ is a prime. Then depending on the choice of primitive element of $\Fp$, one of the following holds true.
\begin{itemize}
\item[(1)] Both $(C_{0}^{8} \cup C_{1}^{8},C_{4}^{8} \cup C_{5}^{8})$ and $(C_{0}^{8} \cup C_{7}^{8},C_{3}^{8} \cup C_{4}^{8})$ are $(\frac{p-1}{4}, \frac{p-1}{4})$-$\NF{\Fp}{\frac{p-1}{16}}$.
\item[(2)] Both $(C_{0}^{8} \cup C_{5}^{8},C_{1}^{8} \cup C_{4}^{8})$ and $(C_{0}^{8} \cup C_{3}^{8},C_{4}^{8} \cup C_{7}^{8})$ are $(\frac{p-1}{4}, \frac{p-1}{4})$-$\NF{\Fp}{\frac{p-1}{16}}$.
\end{itemize}
\end{thm}

\begin{rem}
The constructions in Theorem~\ref{thm-main} and Result~\ref{res-prior}(2) generate $\lambda$-fold near-factorizations in distinct cyclic groups $\Fp$, except for $p=17$.  
\end{rem}

We remark that all previous cyclotomic constructions recorded in Result~\ref{res-prior} satisfy that both subsets of the $\lambda$-fold near-factorizations are single cyclotomic cosets. In contrast, Theorem~\ref{thm-main} generates $\lambda$-fold near-factorizations where both subsets are formed by unions of cyclotomic cosets.

The rest of this paper is organized as follows. In Section~\ref{sec-prelim}, we introduce necessary background concerning cyclotomic numbers. Our main construction is presented in Section~\ref{sec-construction}. Section~\ref{sec-conclusion} concludes the paper.

\section{Preliminaries}
\label{sec-prelim}

Let $N>1$ be an integer. Let $p$ be a prime with $p \equiv 1 \pmod N$. Let $\alpha$ be a fixed primitive element of $\Fp$. Set $f=\frac{p-1}{N}$. For $0 \le i, j \le N-1$, define the \emph{order $N$ cyclotomic number} $(i, j)_{N,\alpha}$ to be the number of $(u, v)$ pairs satisfying $0 \leq u, v \leq f-1$ and 
$$
1+\alpha^{Nu+i}=\alpha^{Nv+j}.
$$
Note that the order $N$ cyclotomic numbers depend on the prime $p$ and the choice of primitive element $\alpha$. To describe the cyclotomic numbers of order 8, which will serve as a key tool in our construction, we first need the following lemma.
\begin{lemma}
\label{lem-Diophantine}
Let $p$ be a prime with $p \equiv 1\pmod8$. Then the following holds true.
\begin{itemize}
\item[(1)] Up to the sign of $y$, there is exactly one pair of integers $(x, y)$ such that $p=x^{2}+4y^{2}$ and $x \equiv 1 \pmod4$.
\item[(2)] Up to the sign of $b$, there is exactly one pair of integers $(a, b)$ such that $p=a^{2}+2b^{2}$ and $a \equiv 1\pmod4$.
\end{itemize}
\end{lemma}

\begin{proof} 
According to \cite[Lemma 3.0.1]{BEW} and \cite[Chapter 1, Section 1]{Cox}, there exist at most and at least one pair of positive integers $\left(x_{0}, y_{0}\right)$, respectively, such that $p=x_{0}^{2}+4y_{0}^{2}$. Therefore, there exists a unique pair of positive integers $\left(x_{0}, y_{0}\right)$ such that $p=x_{0}^{2}+4 y_{0}^{2}$. Clearly, $x_{0}$ is odd. Set $x= \pm x_{0}$ and $y= \pm y_{0}$. The condition $x \equiv 1\pmod4$ determines the sign of $x$. Thus, there is a unique pair $(x, y)$, up to the sign of $y$, such that $p=x^{2}+4y^{2}$. This completes the proof of part (1). Part (2) can be proved analogously.
\end{proof} 

Now we are ready to describe cyclotomic numbers of order $8$ over $\Fp$ for prime $p \equiv 1 \pmod{16}$. We will follow the treatment of Lehmer \cite{Lehmer}.  For simplicity, from now on, we use $(i,j)$ to denote the order $8$ cyclotomic number $(i,j)_{8,\alpha}$ with respect to the primitive element $\alpha$.

\begin{lemma}[\rm{\cite[Appendix]{Lehmer}}]
\label{lem-ordereight}
Let $p$ be a prime with $p \equiv 1\pmod{16}$. Then $2 \in C_0^2$. 
\begin{itemize}
\item Suppose up to the sign of $y$, $(x,y)$ is the unique pair of integers such that $p=x^{2}+4y^{2}$ and $x \equiv 1\pmod4$. 
\item Suppose up to the sign of $b$, $(a, b)$ is the unique pair of integers such that $p=a^{2}+2b^{2}$ and $a \equiv 1 \pmod4$. 
\end{itemize}
Let $\alpha$ be a primitive element of $\Fp$ and $(i, j)$, $0 \le i, j \le 7$, be the order $8$ cyclotomic numbers depending on $p$ and $\alpha$. Then the following holds true.
\begin{itemize}
\item[(1)] Depending on $2 \in C_0^4$ or $2 \in C_2^4$, Table~\ref{tab-fifteen} lists fifteen critical order $8$ cyclotomic numbers $(i,j)$. 

\begin{table}[h!]
\centering
\caption{Fifteen critical order $8$ cyclotomic numbers $(i,j)$}
\begin{tabular}{|l|l|l|}
\hline
$64(i,j)$   & \multicolumn{1}{c|}{if $2 \in C_0^4$} &  \multicolumn{1}{c|}{if $2 \in C_2^4$} \\
\hline
$64(0,0)$ & $p-23-18 x-24 a$ & $p-23+6 x$ \\
\hline
$64(0,1)$ & $p-7+2 x+4 a+16 y+16 b$ & $p-7+2 x+4 a$ \\
\hline
$64(0,2)$ & $p-7+6 x+16 y$ & $p-7-2 x-8 a-16 y$ \\
\hline
$64(0,3)$ & $p-7+2 x+4 a-16 y+16 b$ & $p-7+2 x+4 a$ \\
\hline
$64(0,4)$ & $p-7-2 x+8 a$ & $p-7-10 x$ \\
\hline
$64(0,5)$ & $p-7+2 x+4 a+16 y-16 b$ & $p-7+2 x+4 a$ \\
\hline
$64(0,6)$ & $p-7+6 x-16 y$ & $p-7-2 x-8 a+16 y$ \\
\hline
$64(0,7)$ & $p-7+2 x+4 a-16 y-16 b$ & $p-7+2 x+4 a$ \\
\hline
$64(1,2)$ & $p+1+2 x-4 a$ & $p+1-6 x+4 a$ \\
\hline
$64(1,3)$ & $p+1-6 x+4 a$ & $p+1+2 x-4 a-16 b$ \\
\hline
$64(1,4)$ & $p+1+2 x-4 a$ & $p+1+2 x-4 a+16 y$ \\
\hline
$64(1,5)$ & $p+1+2 x-4 a$ & $p+1+2 x-4 a-16 y$ \\
\hline
$64(1,6)$ & $p+1-6 x+4 a$ & $p+1+2 x-4 a+16 b$ \\
\hline
$64(2,4)$ & $p+1-2 x$ & $p+1+6 x+8 a$ \\
\hline
$64(2,5)$ & $p+1+2 x-4 a$ & $p+1-6 x+4 a$ \\
\hline
\end{tabular}
\label{tab-fifteen}
\end{table}

\item[(2)] Table~\ref{tab-relation} describes the relation among order $8$ cyclotomic numbers. As a consequence, up to the sign of $y$ and $b$, the full set of order $8$ cyclotomic numbers $(i, j)$, $0 \leq i, j \leq 7$, are determined.

\begin{table}[h!]
\centering
\caption{Relation among order $8$ cyclotomic numbers $(i,j)$}
\begin{tabular}{|c|c|c|c|c|c|c|c|c|}
\hline
\diagbox[dir=SE,width=3em]{$i$}{$j$} & $0$ & $1$ & $2$ & $3$ & $4$ & $5$ & $6$ & $7$ \\
\hline
$0$ & $(0,0)$ & $(0,1)$ & $(0,2)$ & $(0,3)$ & $(0,4)$ & $(0,5)$ & $(0,6)$ & $(0,7)$ \\
\hline
$1$ & $(0,1)$ & $(0,7)$ & $(1,2)$ & $(1,3)$ & $(1,4)$ & $(1,5)$ & $(1,6)$ & $(1,2)$ \\
\hline
$2$ & $(0,2)$ & $(1,2)$ & $(0,6)$ & $(1,6)$ & $(2,4)$ & $(2,5)$ & $(2,4)$ & $(1,3)$ \\
\hline
$3$ & $(0,3)$ & $(1,3)$ & $(1,6)$ & $(0,5)$ & $(1,5)$ & $(2,5)$ & $(2,5)$ & $(1,4)$ \\
\hline
$4$ & $(0,4)$ & $(1,4)$ & $(2,4)$ & $(1,5)$ & $(0,4)$ & $(1,4)$ & $(2,4)$ & $(1,5)$ \\
\hline
$5$ & $(0,5)$ & $(1,5)$ & $(2,5)$ & $(2,5)$ & $(1,4)$ & $(0,3)$ & $(1,3)$ & $(1,6)$ \\
\hline
$6$ & $(0,6)$ & $(1,6)$ & $(2,4)$ & $(2,5)$ & $(2,4)$ & $(1,3)$ & $(0,2)$ & $(1,2)$ \\
\hline
$7$ & $(0,7)$ & $(1,2)$ & $(1,3)$ & $(1,4)$ & $(1,5)$ & $(1,6)$ & $(1,2)$ & $(0,1)$ \\
\hline
\end{tabular}
\label{tab-relation}
\end{table}
\end{itemize}
\end{lemma}
\begin{proof}
Since $p \equiv 1\pmod{16}$, $\left( \frac{2}{p} \right)=(-1)^{\frac{p^2-1}{8}}=1$ and $2 \in C_0^2$. The uniqueness of pairs $(x, y)$ and $(a, b)$ follows from Lemma~\ref{lem-Diophantine}. Parts (1) and (2) follows exactly from \cite[Appendix]{Lehmer}.
\end{proof}

Now, we apply Lemma~\ref{lem-ordereight} to the specific prime numbers that are involved in our construction.

\begin{lemma}
\label{lem-primitive}
Let $u$ be a nonzero square such that $p=4u^{2}+12u+1$ is a prime. Let $\alpha$ be a primitive element of $\Fp$.
\begin{itemize}
\item[(1)] Suppose $u$ is even. Then $2 \in C_{0}^{4}$. Moreover,
\begin{itemize}
\item[(1a)] There exists a unique pair of integers $(x,y)=(-2u+1, \pm2u^{\frac{1}{2}})$, up to the sign of $y$, such that $p=x^{2}+4y^{2}$.
\item[(1b)] There exists a unique pair of integers $(a,b)=(2u+1, \pm2u^{\frac{1}{2}})$, up to the sign of $b$, such that $p=a^{2}+2b^{2}$.
\end{itemize}
\item[(2)] Suppose $u$ is odd. Then $2 \in C_{2}^{4}$. Moreover,
\begin{itemize}
\item[(2a)] There exists a unique pair of integers $(x,y)=(2u-1, \pm2u^{\frac{1}{2}})$, up to the sign of $y$, such that $p=x^{2}+4y^{2}$.
\item[(2b)] There exists a unique pair of integers $(a,b)=(-2u-1, \pm2u^{\frac{1}{2}})$, up to the sign of $b$, such that $p=a^{2}+2b^{2}$.
\end{itemize}
\end{itemize}
\end{lemma}
\begin{proof}
The uniqueness of pairs of integers in parts (1) and (2) follows from Lemma~\ref{lem-Diophantine}. In both parts, since $p \equiv 1 \pmod {16}$, by \cite[Theorem 7.5.1]{BEW}, $2 \in C_{0}^{4}$ if and only if $8 \mid 2y$. Since $y= \pm 2 u^{\frac{1}{2}}$, then 
$$
2 \in \begin{cases}
                  C_{0}^{4} & \mbox{if $u$ even,} \\ 
                  C_{2}^{4} & \mbox{if $u$ odd.}
       \end{cases}
$$
\end{proof}

\begin{rem}
\label{rem-primitive}
Let $U_p$ be the set of all primitive elements of $\Fp$. Fix $\alpha \in U_p$, according to Lemma~\ref{lem-primitive}, $x$ and $a$ are determined by $u$, whereas $y$ and $b$ depend on both $u$ and $\alpha$. Thus, we can write $y$ and $b$ as $y_{\alpha}$ and $b_{\alpha}$.  By Lemma~\ref{lem-primitive}, $y_{\alpha}=\pm b_{\alpha}$.

Define $U_p^+=\{ \alpha \in U_p \mid y_{\alpha}=b_{\alpha} \}$ and $U_p^-=\{ \alpha \in U_p \mid y_{\alpha}=-b_{\alpha} \}$. We will show that both $U_p^+$ and $U_p^-$ are nonempty. Since $U_p^+ \cup U_p^-=U_p \ne \es$, without loss of generality, we can assume $U_p^+ \ne \es$. Choose $\alpha \in U_p^+$. Then $y_{\alpha}=b_{\alpha}$.  Set $e=\gcd(8,f)$. Consider the set $\{ 8x+5 \mid 0 \le x \le f-1 \}$. When regarded as modulo $\frac{f}{e}$, the set $\{ 8x+5 \mid 0 \le x \le f-1 \}$ covers each element of $\Z_{\frac{f}{e}}$ exactly $e$ times. Therefore, there exists some $8y+5 \in \{ 8x+5 \mid 0 \le x \le f-1 \}$ such that $\gcd(8y+5,p-1)=1$. Thus, $\alpha^{8y+5} \in U_p$. Note that for $0 \le i,j \le 7$, $(i,j)_{8,\alpha}=(5i,5j)_{8,\alpha^{8y+5}}$. Consequently, $(0,2)_{\alpha}=(0,2)_{\alpha^{8y+5}}$
and $(0,1)_{\alpha}=(0,5)_{\alpha^{8y+5}}$. By Table~\ref{tab-fifteen}, these two equations imply that $y_{\alpha}=y_{\alpha^{8y+5}}$ and $b_{\alpha}=-b_{\alpha^{8y+5}}$. Hence, $y_{\alpha^{8y+5}}=-b_{\alpha^{8y+5}}$ and $\alpha^{8y+5} \in U_p^-$. Thus, both $U_p^+$ and $U_p^-$ are nonempty. Moreover, $U_p^-=\{ \alpha^{8y+5} \mid \alpha \in U_p^+ \}$.
\end{rem}

\section{Main construction}
\label{sec-construction}

In this section, we present our main construction. For simplicity, when writing the order $8$ cyclotomic cosets $C_{\ell}^8$ and the order $8$ cyclotomic numbers $(i,j)$, we always assume that the entries $\ell$, $i$, $j$ are regarded modulo $8$. 

\begin{lemma}
\label{lem-cyceight}
Let $p$ be a prime with $p \equiv 1 \pmod{16}$. 
\begin{itemize}
\item[(1)] For $0 \le i,j \le 7$ with $i \ne j$, $C_i^8C_j^8=\sum_{\ell=0}^7 (j-i,\ell-i) C_{\ell}^8$. 
\item[(2)]Let $S=C_{0}^{8}+C_{1}^{8}$ and $T=C_{4}^{8}+C_{5}^{8}$. Then $ST=\sum_{\ell=0}^{7} v_{\ell} C_{\ell}^8$, where $v_{\ell}=(4,\ell)+(5,\ell)+(3,\ell-1)+(4,\ell-1)$. Moreover, $v_{\ell}=v_{\ell+4}$ for $0 \le \ell \le 3$.
\item[(3)]Let $S=C_{0}^{8}+C_{7}^{8}$ and $T=C_{3}^{8}+C_{4}^{8}$. Then $ST=\sum_{\ell=0}^{7} x_{\ell} C_{\ell}^8$, where $x_{\ell}=(3,\ell)+(4,\ell)+(4,\ell+1)+(5,\ell+1)$. Moreover, $x_{\ell}=x_{\ell+4}$ for $0 \le \ell \le 3$.
\item[(4)]Let $S=C_{0}^{8}+C_{5}^{8}$ and $T=C_{1}^{8}+C_{4}^{8}$. Then $ST=\sum_{\ell=0}^{7} y_{\ell} C_{\ell}^8$, where $y_{\ell}=(1,\ell)+(4,\ell)+(4,\ell+3)+(7,\ell+3)$. Moreover, $y_{\ell}=y_{\ell+4}$ for $0 \le \ell \le 3$.
\item[(5)]Let $S=C_{0}^{8}+C_{3}^{8}$ and $T=C_{4}^{8}+C_{7}^{8}$. Then $ST=\sum_{\ell=0}^{7} z_{\ell} C_{\ell}^8$, where $z_{\ell}=(4,\ell)+(7,\ell)+(1,\ell-3)+(4,\ell-3)$. Moreover, $z_{\ell}=z_{\ell+4}$ for $0 \le \ell \le 3$.
\end{itemize}
\end{lemma}
\begin{proof}
(1) Note that $p \equiv 1 \pmod{16}$ implies $-1 \in C_{0}^8$. Then for each $0 \le i \le 7$, $a \in C_i^8$ if and only if $-a \in C_i^8$. Thus, for $0 \leq i, j \leq 7$ with $i \neq j$, $C_{i}^8C_{j}^8$ does not contain the zero element of $\Fp$. For $h \in \Fp$, we use $[C_i^8C_j^8]_h$ to denote the coefficient of $h$ in the group ring element $C_i^8C_j^8$. For $a \in C_i^8$, $b \in C_j^8$, $h \in C_{\ell}^8$, we have $a+b=h$ if and only if $ac+bc=hc$ for each $c \in C_0^8$. Consequently, $[C_i^8C_j^8]_h$ remains a constant as $h$ ranges over $C_{\ell}^8$. Therefore, for $i \ne j$, we can write $C_{i}^8 C_{j}^8=\sum_{\ell=0}^{7} u_{\ell} C_{\ell}^8$. Set $f=\frac{p-1}{8}$. Then
\begin{align*}
u_{\ell}|C_{\ell}^8|=fu_{\ell}=&|\{(a,b) \mid a \in C_i^8, b \in C_j^8, a+b \in C_{\ell}^8 \}| \\
                                          =&|\{(a,b) \mid a \in C_0^8, b \in C_{j-i}^8, a+b \in C_{\ell-i}^8 \}| \\
                                          =&f|\{(1,b) \mid b \in C_{j-i}^8, a+b \in C_{\ell-i}^8 \}| \\
                                          =&f(j-i,\ell-i)
\end{align*}
Thus, $u_{\ell}=(j-i,\ell-i)$ and $C_i^8C_j^8=\sum_{\ell=0}^7 (j-i,\ell-i) C_{\ell}^8$. 

\noindent (2) Applying part (1), we have
\begin{align*}
ST =&(C_{0}^8+C_{1}^8)(C_{4}^8+C_{5}^8) \\
      =&\sum_{\substack{i \in\{1,2 \} \\j \in\{4, 5\}}} \sum_{\ell=0}^{7}(j-i,\ell-i) C_{\ell}^8 \\
      =&\sum_{\ell=0}^{7}\Big(\sum_{\substack{i \in\{0,1\} \\ j \in\{4, 5\}}}(j-i, \ell-i)\Big) C_{\ell}^8 \\
      =&\sum_{\ell=0}^{7}((4,\ell)+(5,\ell)+(3,\ell-1)+(4,\ell-1)) C_{\ell}^8 \\
      =&\sum_{\ell=0}^{7} v_{\ell} C_{\ell}^8
\end{align*}
Note that 
\begin{align*}
(i, j)=&|\{ (u,v) \mid 0 \le u, v \le f-1,1+\alpha^{8u+i}=\alpha^{8v+j} \}| \\
      =&|\{(u,v) \mid 0 \le u, v \le f-1,1+\alpha^{-8u-i}=\alpha^{8(v-u)+j-i}\}| \\
      =&(8-i,j-i)
\end{align*}
Then for each $0 \le \ell \le 3$,
\begin{align*}
v_{\ell+4}=&(4,\ell+4)+(5,\ell+4)+(3,\ell+3)+(4,\ell+3) \\
               =&(4,\ell)+(3,\ell-1)+(5,\ell)+(4,\ell-1) \\
               =&v_{\ell}
\end{align*}
Parts (3), (4), (5) can be proved similarly as part (2).
\end{proof}

Now we are ready to present the main construction, which establishes Theorem~\ref{thm-main}.

\begin{thm}
\label{thm-ordereight}
Let $u$ be a nonzero square such that $p=4u^{2}+12u+1$ is a prime. Let $\alpha$ be a primitive element of $\Fp$. For $0 \le i \le 7$. define $C_{i}^{8}=\alpha^{i}\left\langle\alpha^{8}\right\rangle$.
\begin{itemize}
\item[(1)] If $\alpha \in U_p^+$, then the following holds.
\begin{itemize}
\item[(1a)] Define $S=C_{0}^{8}+C_{1}^{8}$ and $T=C_{4}^{8}+C_{5}^{8}$. Then $(S,T)$ is a $(\frac{p-1}{4}, \frac{p-1}{4})$-$\NF{\Fp}{\frac{p-1}{16}}$.
\item[(1b)] Define $S=C_{0}^{8}+C_{7}^{8}$ and $T=C_{3}^{8}+C_{4}^{8}$. Then $(S,T)$ is a $(\frac{p-1}{4}, \frac{p-1}{4})$-$\NF{\Fp}{\frac{p-1}{16}}$.
\end{itemize}
\item[(2)] If $\alpha \in U_p^-$, then the following holds.
\begin{itemize}
\item[(2a)] Define $S=C_{0}^{8}+C_{5}^{8}$ and $T=C_{1}^{8}+C_{4}^{8}$. Then $(S,T)$ is a $(\frac{p-1}{4}, \frac{p-1}{4})$-$\NF{\Fp}{\frac{p-1}{16}}$.
\item[(2b)] Define $S=C_{0}^{8}+C_{3}^{8}$ and $T=C_{4}^{8}+C_{7}^{8}$. Then $(S,T)$ is a $(\frac{p-1}{4}, \frac{p-1}{4})$-$\NF{\Fp}{\frac{p-1}{16}}$.
\end{itemize}
\end{itemize}
\end{thm}
\begin{proof}
We only prove case (1a) as the proof of cases (1b), (2a), (2b) is analogous.

By Lemma~\ref{lem-cyceight}(2), $ST=\sum_{\ell=0}^{7} v_{\ell} C_{\ell}$, where $v_{\ell}=(4,\ell)+(5,\ell)+(3,\ell-1)+(4,\ell-1)$ and it suffices to show that $64v_{\ell}=4p-4$ for $0 \le \ell \le 3$. Below, we will handle the two cases $2 \in C_0^4$ and $2 \in C_2^4$ separately.

If $2 \in C_{0}^{4}$, then by Lemma~\ref{lem-primitive}(1), we have $x+a=2$. Since $\alpha \in U_p^+$, we have $y=b$. Employing Lemma~\ref{lem-ordereight}, we compute $64 v_{\ell}$, $0 \le \ell \le 3$, as follows.
\begin{align*}
64v_{0}=&64((4,0)+(5,0)+(3,7)+(4,7))=64((0,4)+(0,5)+(1,4)+(1,5)) \\
            =&(p-7-2x+8a)+(p-7+2x+4a+16y-16b)+2(p+1+2x-4a) \\
            =&4p-12+4x+4a+16y-16b=4p-4+4(-2+x+a)+16(y-b) \\
            =&4p-4\\
64v_{1}=&64((4,1)+(5,1)+(3,0)+(4,0))=64((1,4)+(1,5)+(0,3)+(0,4)) \\
            =&2(p+1+2x-4a)+(p-7+2x+4a-16y+16b)+(p-7-2x+8a) \\
            =&4p-12+4x+4a-16y+16b=4p-4+4(-2+x+a)+16(b-y) \\
            =&4p-4 \\
64v_{2}=&64((4,2)+(5,2)+(3,1)+(4,1))=64((2,4)+(2,5)+(1,3)+(1,4))\\
            =&(p+1-2x)+(p+1+2x-4a)+(p+1-6x+4a)+(p+1+2x-4a) \\
            =&4p+4-4x-4a=4p-4+4(2-x-a) \\
            =&4p-4 \\
64v_{3}=&64((4,3)+(5,3)+(3,2)+(4,2))=64((1,5)+(2,5)+(1,6)+(2,4)) \\
            =&2(p+1+2x-4a)+(p+1-6x+4a)+(p+1-2x) \\
            =&4p+4-4x-4a=4p-4+4(2-x-a) \\
            =&4p-4
\end{align*}

If $2 \in C_{2}^{4}$, then by Lemma~\ref{lem-primitive}(2), we have $x+a=-2$. Since $\alpha \in U_p^+$, we have $y=b$. Employing Lemma~\ref{lem-ordereight}, we compute $64 v_{\ell}$, $0 \le \ell \le 3$, as follows.
\begin{align*}
64v_{0}=&64((4,0)+(5,0)+(3,7)+(4,7))=64((0,4)+(0,5)+(1,4)+(1,5)) \\
            =&(p-7-10x)+(p-7+2x+4a)+(p+1+2x-4a+16y)\\
              &+(p+1+2x-4a-16y) \\
            =&4p-12-4x-4a=4p-4-4(2+x+a) \\
            =&4p-4 \\
64v_{1}=&64((4,1)+(5,1)+(3,0)+(4,0))=64((1,4)+(1,5)+(0,3)+(0,4)) \\
            =&(p+1+2x-4a+16y)+(p+1+2x-4a-16y) \\
              &+(p-7+2x+4a)+(p-7-10x) \\
            =&4p-12-4x-4a=4p-4-4(2+x+a) \\
            =&4p-4
\end{align*}            
\begin{align*}            
64v_{2}=&64((4,2)+(5,2)+(3,1)+(4,1))=64((2,4)+(2,5)+(1,3)+(1,4)) \\
            =&(p+1+6x+8a)+(p+1-6x+4a)\\
              &+(p+1+2x-4a-16b)+(p+1+2x-4a+16y) \\
            =&4p+4+4x+4a+16y-16b=4p-4+4(2+x+a)+16(y-b) \\
            =&4p-4 \\
64v_{3}=&64((4,3)+(5,3)+(3,2)+(4,2))=64((1,5)+(2,5)+(1,6)+(2,4)) \\
            =&(p+1+2x-4a-16y)+(p+1-6x+4a) \\
              &+(p+1+2x-4a+16b)+(p+1+6x+8a) \\
            =&4p+4+4x+4a-16y+16b=4p-4+4(2+x+a)-16(y-b) \\
            =&4p-4
\end{align*}
Consequently, the case (1a) is complete.
\end{proof}

Note that for the first $10^8$ nonzero squares $u$, the sequence 
$$
(4u^{2}+12u+1)_{\mbox{\scriptsize $1 \le u \le (10^8)^2$, $u$ square}}
$$ 
contains $7565563$ primes. Therefore, a natural question is, does the sequence $(4u^{2}+12u+1)_{\mbox{\scriptsize $u \ge 1$, $u$ square}}$ contain infinitely many primes? If so, Theorem~\ref{thm-ordereight} leads to an infinite family of $\lambda$-fold near-factorizations. The following lemma provides an affirmative answer assuming the Bunyakovsky's conjecture \cite[Conjecture 2.3]{Conrad} holds true.

\begin{prop}
Assuming the Bunyakovsky's conjecture is true, then the sequence $(4n^{4}+12n^{2}+1)_{n \ge 1}$ contains infinitely many primes.
\end{prop}
\begin{proof}
Consider the polynomial $f(x)=4x^{4}+12x^{2}+1 \in \Z[x]$. Note that $f(x)$ satisfies the following:
\begin{itemize}
\item[(1)] the leading coefficient is positive
\item[(2)] $f(x)$ is irreducible over $\Z$ 
\item[(3)] the coefficients of $f(x)$ have greatest common divisor $1$
\item[(4)] $\gcd(f(1),f(2))=\gcd(17,113)=1$, which implies for each prime $p$, $f(n) \ne 0 \pmod p$ for some $n \in \Z_{p}$. 
\end{itemize}
Then Bunyakovsky's conjecture \cite[Conjecture 2.3]{Conrad} implies that $(f(n))_{n \ge 1}=(4n^{4}+12n^{2}+1)_{n \ge 1}$ contains infinitely many primes.
\end{proof}

\section{Conclusion}
\label{sec-conclusion}

In this paper, we present a cyclotomic construction of $\lambda$-fold near-factorizations in cyclic group $\Fp$ where $p=4n^{4}+12n^{2}+1$ is a prime for $n \ge 1$. Unlike previous cyclotomic constructions, which produce $\lambda$-fold near-factorizations with both subsets consisting of single cyclotomic cosets, our construction allows each subset to be formed as a union of cyclotomic cosets. Assuming Bunyakovsky’s conjecture holds, this yields an infinite family of $\lambda$-fold near-factorizations in cyclic groups. Our result highlights the potential of using unions of cyclotomic cosets in the construction of  $\lambda$-fold near-factorizations.

\end{document}